\documentclass[reqno]{amsart}

\usepackage{amssymb}           
\usepackage{url}
\usepackage{mathtools}
\usepackage{enumerate}

\newtheorem{thm}{Theorem}[section]

\newtheorem{lemma}[thm]{Lemma}

\theoremstyle{definition}

\renewcommand{\phi}{\varphi}		

\makeatletter
\@mparswitchfalse  

\newcounter{mnotecount}  
\renewcommand{\themnotecount}{\arabic{mnotecount}} 
\newcommand{\mnote}[1]
{\protect{\stepcounter{mnotecount}}$^{\mbox{\footnotesize  $
      \bullet$\themnotecount}}$\marginpar{\parbox[b]{1.2in}{\raggedright\tiny\em
   \themnotecount:\! #1}}}

\makeatother

\title[Areas of Polygons]{Areas of Triangles and Other Polygons with Vertices from Various Sequences.}

\author{Virginia Johnson}

\address{%
Department of Mathematics,
Columbia College,\goodbreak
Columbia, S.C. 29208, USA
}

\email{vjohnson\char'100columbiasc.edu}
\urladdr{http://www.columbiasc.edu/about/directory?view=employee\&id=263}

\author{Charles K. Cook}

\address{%
Emeritus,
University of South Carolina Sumter,\goodbreak
Sumter, S.C. 29208, USA
}

\email{charliecook29150@aim.com}

\begin{document}

\begin{abstract}
Motivated by Elementary Problem B-1172 \cite{Edwards}, formulas for the areas of triangles and other polygons having vertices with coordinates taken from various sequences of integers are obtained.  
\end{abstract} 

\maketitle
\section{Introduction}\label{intro}

The verification of formulas for the area of triangles having coordinates
$(F_n, F_{n + k})$, $ (F_{n + 2 k}, F_{n + 3 k})$,  and  $(F_{n + 4 k}, F_{n +5
k} )$  for both even and odd values of $ k$ was offered as a problem in \emph{The
Fibonacci Quarterly} \cite{Edwards}.   In addition solvers were asked to find
similar formulas for the Lucas vertices cases. An earlier related problem in
\emph{The Fibonacci Quarterly}  \cite{Opher-Shurki} invited solvers to find
the area of a polygon with consecutive Fibonacci numbers as vertex coordinates beginning with $ (F_1,F_2)$ and  running  through $(F_{2n -1},F_{2n})$. In this paper formulas will be obtained for triangles and
$m$-sided polygons with various sequential elements as vertices with the form 
$(p_n,p_{n+k}),\,\, (p_{n+2k},p_{n+3k}),\,\,\dots ,(p_{n+(2m-2)k},p_{n+(2m-1)k)}$.    Similar, but different, approaches to the topics presented here have been considered in \cite{zvon} and \cite{panda}.

\section{Triangles with Fibonacci Type Coordinates
}
The area formulas for a triangle 
with vertices given by Fibonacci numbers  $(F_n, F_{n+k}), \, (F_{n+2k}, F_{n+3k})$,  and $(F_{n+4k}, F_{n+5k}) $
for both even and odd values of $k$ were indicated in the statement of a problem given in 
\cite{Edwards}.  The Problem Poser invited the solvers to show that the area is 
$$
\frac{5F_k^4L_k}{2} \ \text{ if $k$ is even and  }\  \frac{F_k^2L_k^3}{2} \text{\hspace{.1in} if $k$ is odd} .
$$
A proof should appear in a further issue of \emph{The Fibonacci Quarterly}, but for the purposes of completeness, proof will be provided here and treated as a theorem.

\begin{thm}\label{thm:fibtriangle}
The area formulas for triangles with Fibonacci vertices $(F_n, F_{n+k})$, $(F_{n+2k}, F_{n+3k})$,  and $(F_{n+4k}, F_{n+5k}) $ are
$$
\frac{5F_k^4L_k}{2}\  \text{ if $k$ is even and  }\  \frac{F_k^2L_k^3}{2} \text{ if $k$ is odd} .
$$
\end{thm}

\vspace{.25in}

The signed area of this triangle can be calculated using the determinant formula from vector calculus:\cite{area-tri}
\begin{equation}
A =\dfrac{1}{2} \left| \begin{matrix} 
      F_{n+2k}-F_n& F_{n+3k}-F_{n+k}\\
      F_{n+4k}-F_n& F_{n+5k}-F_{n+k}
       \end{matrix}   \right|\\
      =\dfrac{1}{2} \left| \begin{matrix} 
       F_{n+2k}-F_n& F_{n+3k}-F_{n+k}\\
          F_{n+4k}-F_{n+2k}& F_{n+5k}-F_{n+3k}
       \end{matrix}   \right|.
\end{equation}

 It will be useful to find this determinant for the general case first.   Let $a,b,r$ be real numbers with $r\ne 0$. Define a real valued function on the integers by 
 \begin{equation} \label{recursion}
 f(n)=ar^n+b\dfrac{(-1)^{n+1}}{r^n}.
 \end{equation}
 
 This is the general solution to the recursion 
 $$
f(n)= pf(n-1) +f(n-2)
$$
where $p$ and $r$ are related by
$p= r-\frac1r$.
  
 \begin{lemma}\label{gentir}
 The area, $A$, of triangles with vertices $\left(f(n),f(n+k)\right)$,\\  $(f(n+2k),f(n+3k)),$ and $(f(n+4k),\,f(n+5k)$ is 
 
 \begin{equation}\label{gen}
 A= \frac{ab(-1)^n }{2}\left(r^{k}-\frac1{r^k}\right)^3\left(r^{k}+\frac1{r^k}\right)\left(r^{k}+\frac{(-1)^{k+1}}{r^k}\right).
\end{equation}
\end{lemma}

\begin{proof}
 For the general case, the signed area can be determined using:
 $$
A=\dfrac{1}{2} \left| \begin{matrix}   
       f(n+2k)-f(n)& f(n+3k)-f(n+k)\\
          f(n+4k)-f(n+2k)& f(n+5k)-f(n+3k)
       \end{matrix}   \right|.
$$

Note that all entries in the determinant are of the form $f(m+2k)-f(m)$ so applying (\ref{recursion}) it follows that 

$$
f(m+2k)-f(m)=(r^{2k}-1)\left(\frac{ar^{2m+2k}+(-1)^mb}{r^{m+2k}}\right).
$$

Thus

$$ A=\frac{(r^{2k}-1)^2}{2} \left| \begin{matrix} \dfrac{ar^{2n+2k}+(-1)^nb}{r^{n+2k}} & \dfrac{ar^{2n+4k}+(-1)^{n+k}b}{r^{n+3k}}  \vphantom{\displaystyle\int}\\ \vphantom{\displaystyle \int}
 \dfrac{ar^{2n+6k}+(-1)^nb}{r^{n+4k}} & \dfrac{ar^{2n+8k}+(-1)^{n+k}b}{r^{n+5k}}
 \end{matrix}\right|\\
 $$
 $\vphantom{\displaystyle\lim}\\ \vphantom{\displaystyle \lim}$
$$
=\dfrac{(r^{2k}-1)^2}{2r^{n+3k}r^{n+5k}}
\left| 
\begin{matrix} ar^{2n+3k}+(-1)^nbr^k & ar^{2n+4k}+(-1)^{n+k}b\\
ar^{2n+7k}+(-1)^nbr^k & ar^{2n+8k}+(-1)^{n+k}b
\end{matrix}\right|\\
$$
which upon expanding and applying some judicious factoring and grouping yields
\begin{equation}\label{gen}
 A= \frac{ab(-1)^n }{2}\left(r^{k}-\frac1{r^k}\right)^3\left(r^{k}+\frac1{r^k}\right)\left(r^{k}+\frac{(-1)^{k+1}}{r^k}\right).
\end{equation}

Note that this is the signed area so the $(-1)^n$ factor can be ignored when computing the value for $A$.
\vspace{.15in}
\end{proof}

If $a=b=\frac{1}{\sqrt{5}} $, and $r=\frac{1+\sqrt{5}}{2}$, then 
 $f(n)=F_n$, and
if ${a}=1$, ${b}=-1$, and $r=\frac{1+\sqrt{5}}{2}$ then $f(n)=L_n$.

\begin{equation}\label{k:even}
\text{ $k$ is even } \implies\begin{cases}a\left(r^k-\dfrac{1}{r^k}\right)=a\left(r^k+\dfrac{(-1)^{k+1}}{r^k}\right)=F_k \\ &\\
\left(r^k+\dfrac{1}{r^k}\right)=\left(r^k+\dfrac{(-1)(-1)^{k+1}}{r^k}\right)=L_k
\end{cases} 
\end{equation}
and
\begin{equation}\label{k:odd}
\text{$k$ is odd } \implies \begin{cases}a\left(r^k+\dfrac{1}{r^k}\right)=a\left(r^k+\dfrac{-1_(-1)^{k+1}}{r^k}\right)=F_k, \\ &\\
  \left(r^k-\dfrac{1}{r^k}\right)=\left(r^k+\dfrac{-1(-1)^{k+1}}{r^k}\right)=L_k
\end{cases} 
\end{equation}

We can now prove Theorem \ref{thm:fibtriangle}.
\begin{proof}
If $k$ is even, using (\ref{gen}) and (\ref{k:even}), the  area of the triangle will be:

$$
 A= \frac{ab}{2}\left(r^k-\frac{1}{r^k}\right)^4\left(r^k +\dfrac{1}{r^k}\right)\\
 =\dfrac{a^4}{2a^2}\left(r^k-\dfrac{1}{r^k}\right)^4\left(r^k +\dfrac{1}{r^k}\right)\\
 =\dfrac{5F_k^4L_k}{2}.
$$

For $k$  odd, using (\ref{gen}) and (\ref{k:odd}) the area of the triangle will be
$$
 A= \dfrac{a^2}{2}\left(r^k-\dfrac{1}{r^k}\right)^3\left(r^k +\dfrac{1}{r^k}\right)^2\\
 =\dfrac{F_k^2L_k^3}{2}.
$$
\end{proof}

\begin{thm}\label{thm:luctriangle}
The area formulas for triangles with Lucas vertices $(L_n, L_{n+k})$, $(L_{n+2k}, L_{n+3k})$,  and $(L_{n+4k}, L_{n+5k}) $ are
$$
\dfrac{25}{2}F_k^4L_k\  \text{ if $k$ is even and  }\  \dfrac{5F_k^2L_k^3}{2} \text{ if $k$ is odd} .
$$
\end{thm}
\begin{proof}

If $k$ is even, using (\ref{gen}) and (\ref{k:even})

$$
 A = \dfrac{\overline{a}\overline{b}}{2}\left(r^k-\dfrac{1}{r^k}\right)^4\left(r^k +\dfrac{1}{r^k}\right)\\ 
 =\dfrac{-a^4}{2a^4}\left(r^k+\dfrac{(-1)^{k+1}}{r^k}\right)^4\left(r^k +\dfrac{1}{r^k}\right).
$$
The absolute value gives the area as  $\dfrac{25}{2}F_k^4L_k$.

If $k$ is odd, using (\ref{gen}) and (\ref{k:odd})

$$
 A = \dfrac{\overline{a}\overline{b}}{2}\left(r^k-\dfrac{1}{r^k}\right)^3\left(r^k +\dfrac{1}{r^k}\right)^2
 =\dfrac{-a^2}{2a^2}\left(r^k+\dfrac{1}{r^k}\right)^2 \left(r^k-\dfrac{1}{r^k}\right)^3.
$$

The absolute value gives the area as $\dfrac{5F_k^2L_k^3}{2}$.
\end{proof}

\vspace{.25in}
We turn our attention to generalized Fibonacci numbers.
Consider the sequence $[G_n]$  where $G_0=t-s$, $G_1=s$, $G_2=t$, and $G_{n+1}=G_n+G_{n-1}$.
Using  (\ref{recursion}) and 
solving the linear system, $f(0)=(t-s)$ and $f(1)=s$, we find that if $a=\dfrac{s+\frac{(t-s)}{r}}{\sqrt{5}}$,  $b=\dfrac{s+(s-t)r}{\sqrt{5}}$,  and  $r=\frac{1+\sqrt{5}}{2}$, then $f(n)=G_n$. \\  Note that $ab=\dfrac{s^2+st-t^2}{5}$.

\begin{thm}\label{thm:gentriangle}
The area formulas for triangles with Generalized Fibonacci Numbers with vertices $(G_n, G_{n+k})$, $(G_{n+2k}, G_{n+3k})$,  and $(G_{n+4k}, G_{n+5k}) $ are
$$
\dfrac{5(s^2+st-t^2)}{2}F_k^4L_k
 \text{ if $k$ is even and  } \left(\dfrac{s^2+st-t^2}{2}\right)F_k^2L_k^3 \text{ if $k$ is odd} .
 $$
\end{thm}

\begin{proof}

If $k$ is even, using (\ref{gen}) and (\ref{k:even}).

$$
 A= \dfrac{ab}{2}\left(r^k-\dfrac{1}{r^k}\right)^4\left(r^k +\dfrac{1}{r^k}\right)=\dfrac{5(s^2+st-t^2)}{2}F_k^4L_k.
$$

\vspace{.1in}

 If $k$ is  odd, using (\ref{gen}) and (\ref{k:odd})

\begin{align*}
 A & = \dfrac{ab}{2}\left(r^k-\dfrac{1}{r^k}\right)^3\left(r^k +\dfrac{1}{r^k}\right)^2\\
&=\dfrac{(s^2+s-t^2)}{2\sqrt{5}^2}\left(r^k+\dfrac{(-1)^{k+1}}{r^k}\right)^2\left(r^k+\dfrac{-1(-1)^{k+1}}{r^k}\right)^3.
\end{align*}

Hence $A=\left(\dfrac{s^2+st-t^2}{2}\right)F_k^2L_k^3$.
\end{proof}
\vspace{.25in}
The   Pell and Pell-Lucas numbers have Binet forms:  
$$
P_n=\frac{(1+\sqrt{2})^n-(1-\sqrt{2})^n}{2\sqrt{2}}=\frac{\alpha^n-\beta^n}{\alpha-\beta}
$$
and
$$
Q_n=(1+\sqrt(2))^n+(1-\sqrt{2})^n=\alpha^n+\beta^n.
$$
[\cite{Bicknell:primer}\cite{Horadam:pell}\cite{Koshy:Pell-book}]

Returning to the general function (\ref{recursion}), $ f(n)=ar^n+b\dfrac{(-1)^{n+1}}{r^n}$;  if $r=1+\sqrt{2}$ and $a=b=\dfrac{1}{2\sqrt{2}}$, then $f(n)=P_n$.
  
If $r=1+\sqrt{2},\,\,a=1 $, and $b=-1$, then $f(n)= Q_n$.

\begin{equation}\label{pell:even}
\text{ $k$ is even } \implies\begin{cases}a\left(r^k-\dfrac{1}{r^k}\right)=a\left(r^k+\dfrac{(-1)^{k+1}}{r^k}\right)=P_k \\ &\\
\left(r^k+\dfrac{1}{r^k}\right)=\left(r^k+\dfrac{(-1)(-1)^{k+1}}{r^k}\right)=Q_k
\end{cases} 
\end{equation}
and
\begin{equation}\label{pell:odd}
\text{$k$ is odd } \implies \begin{cases}a\left(r^k+\dfrac{1}{r^k}\right)=a\left(r^k+\dfrac{-1_(-1)^{k+1}}{r^k}\right)=P_k, \\ &\\
  \left(r^k-\dfrac{1}{r^k}\right)=\left(r^k+\dfrac{-1(-1)^{k+1}}{r^k}\right)=Q_k
\end{cases} 
\end{equation}

\begin{thm}\label{thm:pelltriangle}
The area formulas for triangles with Pell vertices $(P_n, P_{n+k})$, $(P_{n+2k}, P_{n+3k})$,  and $P_{n+4k},P_{n+5k}) $ are
$$
4P_k^4Q_k \text{ if $k$ is even and  } \dfrac{P_k^2Q_k^3}{2} \text{ if $k$ is odd} .
$$
\end{thm}

\begin{proof}
Using (\ref{gen}) and  (\ref{pell:even}),  
if $k$ is even:

$$
 A = \dfrac{ab}{2}\left(r^k-\dfrac{1}{r^k}\right)^4\left(r^k +\dfrac{1}{r^k}\right)=\frac{(2\sqrt{2})^2}{2(2\sqrt{2})^4}\left(r^k-\dfrac{1}{r^k}\right)^4\left(r^k +\dfrac{1}{r^k}\right)=4P_k^4Q_k.
$$

If $k$ is odd, using (\ref{gen}) and (\ref{pell:odd}),

$$
 A  = \dfrac{ab}{2}\left(r^k-\dfrac{1}{r^k}\right)^3\left(r^k +\dfrac{1}{r^k}\right)^2=\dfrac{1}{2(2\sqrt{2})^2}\left(r^k-\dfrac{1}{r^k}\right)^3\left(r^k +\dfrac{1}{r^k}\right)^2=\dfrac{P_k^2Q_k^3}{2}.
 $$
\end{proof}

\begin{thm}\label{thm:lucpelltriangle}
The area formulas for triangles with Lucas-Pell vertices $(Q_n, Q_{n+k})$, $(Q_{n+2k}, Q_{n+3k})$,  and $Q_{n+4k},Q_{n+5k}) $ are
$$
32P_k^4Q_k \text{ if $k$ is even and  } 8P_k^2Q_k^3 \text{ if $k$ is odd} .
$$
\end{thm}
\begin{proof}
If $k$ is even:

$$
 A= \frac{ab}{2}\left(r^k-\frac{1}{r^k}\right)^4\left(r^k +\frac{1}{r^k}\right)=\frac{(2\sqrt{2})^4}{2(2\sqrt{2})^4}\left(r^k-\dfrac{1}{r^k}\right)^4\left(r^k +\dfrac{1}{r^k}\right)=32P_k^4Q_k.
 $$
\
If $k$ is odd:

$$
 A = \dfrac{ab}{2}\left(r^k-\dfrac{1}{r^k}\right)^3\left(r^k +\dfrac{1}{r^k}\right)^2=\dfrac{(2\sqrt{2})^2}{2(2\sqrt{2})^2}\left(r^k-\dfrac{1}{r^k}\right)^3\left(r^k +\dfrac{1}{r^k}\right)^2=8P_k^2Q_k^3.
 $$
\end{proof}

\vspace{.25in}
For triangles with Jacobsthal and Jacobsthal-Lucas \cite{Horadam:Jacobsthal} sequences as vertices, the
eigenvalues do not lend themselves to the structure using (\ref{recursion}).
So using the Jacobsthal vertices 
$(J_n, J_{n+k})$, $(J_{n+2k}, J_{n+3k})$,  and $(J_{n+4k}, J_{n+5k}) $, 
the Jacobsthal-Lucas vertices $(j_n, j_{n+k})$, $(j_{n+2k}, j_{n+3k})$,  
and $(j_{n+4k}, j_{n+5k}) $ and the Binet forms for the Jacobsthal
and Jacobsthal-Lucas numbers \cite{Horadam:Basic} :
$$
J_n=\frac{1}{3}(2^n-(-1)^n)
$$

and

$$
j_n=2^n+(-1)^n
$$

it is seen that the following holds:

\begin{thm}\label{thm:Jacob}
 The  points with Jacobsthal coordinates
 
  $$(J_n, J_{n+k}), (J_{n+2k}, J_{n+3k}), (J_{n+4k}, J_{n+5k}),\dots,(J_{n+2mk},J_{n+(2m+1)k}),$$
  
  are colinear  as are the points with Jacobsthal-Lucas coordinates
  
  $$(j_n, j_{n+k}), (j_{n+2k}, j_{n+3k}), (j_{n+4k}, j_{n+5k}), \cdots , (j_{n+2mk},j_{n+(2m+1)k}).$$
  
\end{thm}
\begin{proof}
 The equation of the line through the first two points with Jacobsthal coordinates is 
 
 $$y=\frac{J_{n+3k}-J_{n+k}}{J_{n+2k}-J_n}\left(x-J_n\right)+J_{n+k}.$$
 
 For $x=J_{n+2mk}$ with $m\geq4$  we have $y=\dfrac{J_{n+3k}-J_{n+k}}{J_{n+2k}-J_n}\left(J_{n+2mk}-J_n\right)+J_{n+k}$.  
 Using the  Binet form, this can be written as 
 
 $$y=\left(\dfrac{2^{n+3k}-2^{n+k}}{2^{n+2k}-2^n}\right)\left(\dfrac{2^{n+2mk}-2^n}{3}\right)+\left(\dfrac{2^{n+k}-(-1)^{n+k}}{3}\right)$$
 which simplifies to $y=\dfrac{2^{n+(2m+1)k}-(-1)^{n+k}}{3}$.  Since $n+k$ and $n+(2m+1)k$ are congruent mod 2, $y=J_{n+(2m+1)k}$.  Thus all the points in the set are colinear.
 The proof for the set of points with Jacobsthal-Lucas coordinates is similar.
\end{proof}
Therefore, there are no polygons with vertices using  Jacobsthal coordinates or Jacobsthal-Lucas coordinates of this form.

\section{Area of $m$-sided polygons with Fibonacci type vertices}

\noindent
Note:  This approach was motivated by problem B-1167 in The Fibonacci Quarterly. \cite{Opher-Shurki} 
\bigskip

\noindent
Problem Restated:

What is the area of an $m$-sided polygons whose vertices have coordinates \\
$(F_n,F_{n+k}),\, (F_{n+2k},F_{n+3k}),\,(F_{n+4k},F_{n+5k}); \dots, (F_{n+(m-2)k},F_{n+(m-1)k})$?\\

\bigskip

\noindent

Using the  surveyor's  formula  \cite{shoelace} for the area of polygons with vertex coordinates\\  $P_1(x_1,y_1),\,P_2(x_2,y_2),\, \dots\,, P_n(x_n,y_n)$ : 

\begin{equation}\label{shoe}
\frac{1}{2}\left|\sum_{i=1}^{n-1}x_iy_{i+1}+x_ny_1-\sum_{i=1}^{n-1}x_{i+1}y_i-x_1y_n\right|
\end{equation}
we will start with a very  general case as before.   Let $a,b,r$ be real numbers with $r\ne 0$. Define a real valued function on the integers as before (\ref{recursion}):
 
 $$f(n)=ar^n+b\dfrac{(-1)^{n+1}}{r^n}.$$

 Using equation (\ref{shoe}), the area of an $m$-sided polygon can be expressed as 
\begin{align*}
A=\frac{1}{2}\Big(&f(n)f(n+3k)+f(n+2k)f(n+5k)+\dots \\ 
&+f(n+(2m-4)k)f(n+(2m-1)k)\\
&+f(n+(2m-2)k))fn(n+k)\\
&-f(n+2k)f(n+k)-f(n+4k)f(n+3k)-\dots \\
&-f(n+(2m-2)k)f(n+(2m-3)k)-f(n)f(n+(m-1)k)\Big).
\end{align*}

Reordering the terms, we have 
\begin{align*}
A=\frac12& \Big(([f(n)f(n+3k)-f(n+2k)f(n+k)]\\
&+[f(n+2k)f(n+5k)-f(n+4k)f(n+3k)]\dots \\ 
&+[f(n+(2m-4)k)f(n+(2m-1)k- f(n+(2m-2)k)f(n+(2m-3)k)]\\  
&+[f(n+(2p-2)k))f(n+k)-f(n)f(n+(2p-1)k)] \big) .
\end{align*}

Note that  the first $(m-1)$  pairs  have the  form:

 $f(t)f(t+3k)-f(t+k)f(t+2k)$.  Applying equation (\ref{recursion}), 

\begin{align*}
&f(t)f(t+3k)-f(t+k)f(t+2k)\\
&=\left(ar^t+\frac{b(-1)^{t+1}}{r^t}\right)\left(ar^{t+3k}+\frac{b(-1)^{t+k+1}}{r^{t+3k}}\right)\\
&-\left(ar^{t+k}+\dfrac{b(-1)^{t+k+1}}{r^{t+k}}\right)\left(ar^{t+2k}+\frac{b(-1)^{t+1}}{r{t+2k}}\right)\\
&=ab(-1)^{t+1}\left[\frac{(-1)^{k}}{r^{3k}}+r^{3k}-\frac{1}{r^k}-(-1)^{k}\right].
\end{align*}
If  $k$ odd, this factors as 

 $$ab(-1)^{t+1}\left(r^k+\frac{1}{r^k}\right)\left(r^{2k}-\frac{1}{r^{2k}}\right).$$

For even $k$, it factors as

  $$ab(-1)^{t+1}\left(r^k-\frac{1}{r^k}\right)\left(r^{2k}-\frac{1}{r^{2k}}\right).$$

There will be $m-1$ of these terms.   

The last pair in the formula is as follows:
\begin{align*}
&f(n+(2m-2)k))f(n+k)-f(n+(2m-1)k)f(n)\\
&=\left(ar^{n+(2m-2)k}+\dfrac{b(-1)^{n+1}}{r^{n+(2m-2)k}}\right)\left(ar^{n+k}+\frac{b(-1)^{n+k+1}}{r^{n+k}}\right)\\
&-\left(ar^{n+(2m-1)k}+\dfrac{b(-1)^{n+k+1}}{r^{n+(2m-1)k}}\right)\left(ar^n\frac{b(-1)^{n+1}}{r^n}\right)\\
&=ab(-1)^n\left[(-1)^{k+1}r^{(2m-3)k}-\frac{(1}{r^{(2m-3)k}}+r^{(2m-1)k}-\frac{(-1)^{k+1}}{r^{(2m-1)k}}\right].
\end{align*}

For $k$  odd this factors to 

$$(r^k+\frac{1}{r^k})(r^{(2m-2)k}-\frac{1}{r^{(2m-2)k}}).$$

For $k$ even we have

 $$(r^k-\frac{1}{r^k})(r^{(2m-2)k}-\frac{1}{r^{(2m-2)k}}).$$

Therefore, the formula for the area of a polygon with $m$ sides and vertices 

$(f(n),f(n+k)), (f(n+2k),f(n+3k)+\cdots+(f(n+(2m-2)k),f(n+(2m-1)k)$  with $k$ even will be:

\begin{equation}\label{gen:even}
\left|\frac{1}{2}ab\left[(m-1)\left(r^k-\frac{1}{r^k}\right)\left(r^{2k}-\frac{1}{r^{2k}}\right)-
\left(r^k-\frac{1}{r^k}\right)\left(r^{(2m-2)k}-\frac{1}{r^{(2m-2)k}}\right)\right]\right|
\end{equation}

and for  $k$ odd:
\begin{equation}\label{gen:odd}
\left|\frac{1}{2}ab\left[(m-1)\left(r^k+\frac{1}{r^k}\right)\left(r^{2k}-\frac{1}{r^{2k}}\right)-
\left(r^k+\frac{1}{r^k}\right)\left(r^{(2m-2)k}-\frac{1}{r^{(2m-2)k}}\right)\right]\right|.
\end{equation}

\begin{thm}
The area for any $m$-sided polygon with Fibonacci coordinates as described   will be:

$$\frac{1}{2}\left|(m-1)F_kF_{2k}-F_kF_{(2m-2)k}\right|.$$

The area for any $m$-sided polygon with Lucas coordinates as described will be:

$$\frac{5}{2}\left|(m-1)F_kF_{2k}-F_kF_{(2m-2)k}\right|.$$
\end{thm}

The proof follows directly from equations (\ref{k:even}), ({\ref{k:odd}), (\ref{gen:even}) and (\ref{gen:odd}).

\vspace{.3in}

\begin{thm}
The area for any $m$-sided polygon with Pell coordinates as described will be:

$$\frac{1}{2}\left|(m-1)P_kP_{2k}-P_kP_{(2m-2)k}\right|.$$

The area for any $m$-sided polygon with Pell-Lucas coordinates as described will be:

$$4\left|(m-1)P_kP_{2k}-P_kP_{(2m-2)k}\right|.$$
\end{thm}

 The proof follows directly from equations (\ref{pell:even}), ({\ref{pell:odd}), (\ref{gen:even}) and (\ref{gen:odd}).

\begin{thm}
The area for any $m$-sided polygon with Generalized Fibonacci coordinates as described in Section 3 will be:

$$\frac{s^2+st-t^2}{2}\left|(m-1)F_kF_{2k}-F_kF_{(2m-2)k}\right|.$$

\end{thm}

 The proof follows directly from equations (\ref{k:even}), ({\ref{k:odd}), (\ref{gen:even}) and (\ref{gen:odd}).

\section{Triangles with Polygonal Number Co-ordinates}
The general polygonal numbers are defined by the formula 
$$
P_n^{(r)}=\frac{n(n(r-2)-(r-4))}{2},  \text{  for  }r\geq3.
$$
\cite{Deza-Deza}
\begin{thm}\label{thm:poly} For  any triangle with polygonal number vertices $(P_n, P_{n+k})$, $(P_{n+2k}, P_{n+3k})$,  and 
 $(P_{n+4k}, P_{n+5k}) $,  the area is
 \begin{equation}\label{poly_area}
 A= 4(r - 2)^2 k^4.
 \end{equation}
 \end{thm}
\begin{proof}
Analogous to the proof of the Fibonacci case, the area is the absolute value of the
determinant
 $$
 A =\dfrac{1}{2} \left| \begin{matrix} 
       P_{n+2k}-P_n& P_{n+3k}-P_{n+k}\\
          P_{n+4k}-P_n& P_{n+5k}-P_{n+k}
       \end{matrix}   \right|
$$
which upon expansion yields the following
$$
A=\frac{1}{2}\left((P_{n+2k}-P_n)(P_{n+5k}-P_{n+k})-(P_{n+4k}-P_n)(P_{n+3k}-P_{n+k})\right)
$$
$$
= \frac{1}{2}[(k[(r-2)(2n+2k)-(r-4)])(2k[(r-2)(2n+6k)-(r-4)])
$$
$$-(k[(r-2)(2n+4k)-(r-4)])(2k[(r-2)(2n+4k)-(r-4)])]
$$
$$=\frac{1}{2}[(2k^2[(r-2)^2(4n^2+16nk+12k^2)+(r-4)^2-(r-4)(r-2)(4n+8k)])
$$
$$-(2k^2[(r-2)^2(4n^2+16nk+16k^2)+(r-4)^2-(r-4)(r-2)(4n+8k)])]=-\frac{8k^4(r-2)^2}{2}
$$
The absolute value yields $A=4(r-2)^2k^4$.
\end{proof}
So,  it follows that

Triangles with Triangular Number Co-ordinates have area $A=4k^4$.

Triangles with Square Number Co-ordinates have area $A=16k^4$.

Triangles with Pentagonal Number Co-ordinates have area $A=36k^4$.

Triangles with Hexagonal Number Co-ordinates have area $A=64k^4$.

Triangles with Heptagonal Number Co-ordinates have area $A=100k^4$.

\vspace{.2in}

\section{Area of $m$-sided polygons with Polygonal Number Co-ordinates}
The following table summarizes  results for the area of triangles, quadrilaterals, pentagons, hexagons, and heptagons using  polygonal vertices in the given pattern.

\vspace{.2in}
\begin{center}
 
\begin{tabular}{| c | c | c | c | c | c |}
\hline
$m$ & Triangular & Square & Pentagonal & Hexagonal & Heptagonal\\
\hline
3 & $4k^4$ & 1$6k^4$ & $36k^4$ & $64k^4$ & $ 100k^4$\\
\hline
4 & $ 16 k^4$  & $64k^4$  & $144k^4$  & $256k^4$  & $400k^4$ \\
\hline
5 & $40k^4$  & $160k^4$  & $360k^4$  & $640k^4$  & $1000k^4$ \\
\hline
6 & $ 80k^4$  & $320k^4$  & $ 720k^4$  & $1280k^4$  & $2000k^4$ \\
\hline
7 & $140k^4$  & $560k^4$  & $ 1260k^4$  & $ 2240 k^4$  & $ 3500k^4$ \\
\hline
\end{tabular}

\end{center}

\vspace{.2in}

Note that the coefficients of the $m$-sided polygon is 4 times the sequence\\  $\{1,4,10,20,35,...\}$
which can be written as $ \dfrac{m(m -1)(m -2)}{6}$ and is the general term for the tetrahedral (or
triangle pyramid) number sequence having OEIS number A000292. Thus it is reasonable to
conjecture the following:

\begin{thm}
 The area of any m-sided polygon ($m\geq3$)with polygonal number vertices 
 $$
 (P_n, P_{n + k}), (P_{n + 2 k}, P_{n + 3 k}), \dots (P_{n +(2m-2) k}, P_{n +(2m-1) k})
 $$ 
  is $\dfrac{4m(m - 1)(m - 2)(r - 2)^2 k^4}{6}$ .
\end{thm}

\begin{proof} 

For the general case, it is convenient
to re-write (\ref{shoe}) as

\begin{equation}\label{shoe2}
\frac{1}{2}\left|\sum_{i=1}^{m-1}x_iy_{i+1}+x_my_1-\sum_{i=1}^{m-1}x_{i+1}y_i-x_1y_m\right|.
\end{equation}

The proof is by  induction on $m$.  The statement has been shown to be true for $m=3.$ Assume that the area of a polygon with $m=t$ is $\dfrac{4t(t- 1)(t- 2)(r - 2)^2 k^4}{6}$.  
Note that the area of the $t+1$ sided polygons will be the area of the $t$ sided polygon plus the area of the triangle with coordinates $(P_n,\,P_{n+k}),\,(P_{n+(2t-2)k},\,P_{n+(2t-1)k})$, and$(P_{2t},\,P_{n+(2t+1)k})$
which is 
 $$
 A =\dfrac{1}{2} \left| \begin{matrix} 
       P_{n+(2t-2)k}-P_n& P_{n+(2t-1)k}-P_{n+k}\\
          P_{n+(2t)k}-P_n& P_{n+(2t+1)k}-P_{n+k}
       \end{matrix}   \right|.
$$
On expansion, this yields
$$A=\frac{1}{2}\left((P_{n+(2t-2)k}-P_n)(P_{n+(2t+1)k}-P_{n+k})-( P_{n+(2t-1)k}-P_{n+k})( P_{n+(2t-1)k}-P_{n+k})\right)$$
$$=\frac{1}{2}[(t-1)k[2(r-2)((t-1)k+n)-(r-4)]kt[2(r-2)((t+1)k+n)-((r-4)]$$
$$-k(t-1)[2(r-2)(tk+n)-(r-4)]kt[2(r-2)(tk+n)-(r-4)]]$$
$$=\frac{1}{2}k^2t(t-1)[[2(r-2)((t-1)k+n)-(r-4)][2(r-2)((t+1)k+n)-((r-4)]$$
$$-[2(r-2)(tk+n)-(r-4)][2(r-2)(tk+n)-(r-4)]]$$
$$=\frac{1}{2}k^2t(t-1)(-4(r-2)^2k^2=\frac{-4(r-2)^2k^4(t-1)(3t)}{6}.$$

The absolute value yields 
$$\frac{4(r-2)^2k^4(t-1)(3t)}{6}.$$
Adding this area to the $t$-sided polygon gives the area of the $t+1$-sided polygon as 
$$\frac{4(t+1)(t)(t-1)(r-2)^2k^4}{6}$$ as required.

\end{proof}

\vspace{.2in}

\section{Concluding Comments}

Whereas the second order sequences seem to lend themselves to obtainable formulas for
polygonal areas, the third order sequences do not appear to be as obliging. For example, three popular third order sequences yield the following areas of triangle with the indicated sequential coordinates:
Area of triangles with $n=1$ and the indicated value for $k$:
\begin{center}
  \begin{tabular}{ |  l  |  l  |  l  |  l  | }

    \hline
      &  &  &  \\
    $k$ & Tribonacci  &  Perrin                &  Padovan \\   
      &                       &                            &  \\ \hline 
           &  &  &  \\
    1 & 3                    & $\dfrac{9}{2}$     & 0 \\ 
         &                       &                            &  \\ \hline 
           &  &  &  \\
    2 & 64                  & $\dfrac{47}{2}$   & 1 \\ 
     &                       &                            &  \\ \hline 
       &  &  &  \\
   3 & $849=3(283)$ & $\dfrac{31}{9}$   &$15=3\cdot 5$\\
    &                       &                            &  \\ \hline 
      &  &  &  \\
4  &  $23360=5\cdot64\cdot73$  &  $\dfrac{298}{2}=149$ & $44=4\cdot11$\\
 &                       &                            &  \\ \hline 
   &  &  &  \\
5 &  $509729 = 11\cdot149\cdot311$ & $\dfrac{1629}{2} = \dfrac{9\cdot181}{2}$ &  $95 = 5\cdot19$\\ 
 &                       &                            &  \\ \hline 
   &  &  &  \\
6 & $10049160 = 3\cdot5\cdot8\cdot11\cdot23\cdot331$ & $ \dfrac{9640}{2} = \dfrac{5\cdot8\cdot241}{2} = 4820$ &  $810 = 2\cdot5\cdot81$\\
 &                       &                            &  \\ \hline
    \hline
  \end{tabular}
\end{center}

Comparing the Perrin numbers with the Padovan numbers yields no patterns. Nor does a
comparison of the Tribonacci numbers with themselves. We leave the investigation of the
quadrilateral cases to the interested reader.

\providecommand{\bysame}{\leavevmode\hbox to3em{\hrulefill}\thinspace}
\providecommand{\MR}{\relax\ifhmode\unskip\space\fi MR }
\providecommand{\MRhref}[2]{%
  \href{http://www.ams.org/mathscinet-getitem?mr=#1}{#2}
}
\providecommand{\href}[2]{#2}

\end{document}